\documentclass[11pt]{article}
\usepackage[a4paper,margin=28mm]{geometry}
\usepackage[T1]{fontenc}
\usepackage{lmodern}
\usepackage{amsmath,amssymb,amsthm}
\usepackage{microtype,enumitem,booktabs}
\usepackage{cite}
\usepackage[colorlinks=true,linkcolor=blue,citecolor=blue,urlcolor=blue]{hyperref}

\newcommand{\DT}{\operatorname{Destop}}
\newcommand{\AT}{\operatorname{Asctop}}
\newcommand{\Peak}{\operatorname{Peak}}
\newcommand{\PW}{\operatorname{Pword}}
\newcommand{\rev}{\operatorname{rev}}
\renewcommand{\S}{\mathcal{S}}
\newcommand{\A}{\mathcal{A}}
\newcommand{\eps}{\varepsilon}
\newtheorem{theorem}{Theorem}[section]
\newtheorem{lemma}[theorem]{Lemma}

\newtheorem{corollary}[theorem]{Corollary}
\newtheorem{conjecture}[theorem]{Conjecture}
\theoremstyle{remark}
\newtheorem{remark}[theorem]{Remark}
\theoremstyle{definition}
\newtheorem{example}[theorem]{Example}
\numberwithin{equation}{section}
\setlist[enumerate]{label=\textup{(\roman*)},leftmargin=*,itemsep=3pt,topsep=5pt}
\begin{document}
\begin{center}
{\bf \large Two refined Wilf-equivalences for patterns of length 4 and a conjecture on Dumont permutations of the first kind}
\vspace{0.3cm}

\end{center}
\begin{center}

 Zuo-Ru Zhang, \ Hongkuan Zhao$^{*}$\\[8pt]
School of Mathematical Science\\
Hebei Normal University\\
 Shijiazhuang 050024, P. R. China\\[6pt]

{\tt  zrzhang@hebtu.edu.cn}

{\tt  hkzhao2001@163.com}

\end{center}

\begingroup
\renewcommand{\thefootnote}{\fnsymbol{footnote}}
\footnotetext[1]{Corresponding author.}
\endgroup

\begin{abstract}
We construct a bijection between $2143$-avoiding and $3421$-avoiding permutations that preserves the descent top set and the ascent top set and reverses the left-to-right order of the peaks.
This bijection allows us to establish a bijection between $1243$-avoiding and $3412$-avoiding permutations with the same properties.
These results refine two cases of a conjecture posed by Burstein and prove a conjecture posed by Burstein and Jones.
\end{abstract}
\vspace{0.3cm}


\section{Introduction}
Let $\S_n$ denote the set of permutations of $[n]=\{1,\ldots,n\}$.
We write a permutation $\pi\in \S_n$ as $\pi=\pi_1\pi_2\ldots\pi_n$.
Given two permutations $\sigma\in\S_n$ and $\tau\in\S_k$ with $n\ge k$, we say that $\sigma$ contains the \emph{pattern} $\tau$ if there exist integers $i_1<i_2<\cdots<i_k$ such that $\sigma_{i_1}\sigma_{i_2}\cdots\sigma_{i_k}$ is order isomorphic to $\tau$.
Otherwise, we say that $\sigma$ \emph{avoids} $\tau$, or $\sigma$ is $\tau$-avoiding.
We write $\S_n(\tau)$ for the set of $\tau$-avoiding permutations in $\S_n$.
The study of pattern avoidance can be traced back to the work of MacMahon~\cite{MacMahon}.
For over a half century, this theme of research has been the focus in enumerative combinatorics (see B\'{o}na's book~\cite{Bona} and Kitaev's book~\cite{Kitaev}).

Two patterns $\sigma$ and $\tau$ are \emph{Wilf-equivalent} if $|\S_n(\sigma)|=|\S_n(\tau)|$ for every $n$.
A \emph{permutation statistic} is defined to be a function st: $\S_n \rightarrow A$, where $A$ is some set.
For permutation statistics $\mathrm{st}_1,\ldots,\mathrm{st}_r$, $\sigma$ and $\tau$ are said to be \emph{$(\mathrm{st}_1,\ldots,\mathrm{st}_r)$-Wilf-equivalent} if these statistics have the same joint distribution on $\S_n(\sigma)$ and $\S_n(\tau)$ for every $n$.
See~\cite{Bloom,BS1,BS2,Chan,CL,CK,DRS,pps,eli,RSZ,YGZ,YWZ,ZY1,ZY2} for various refinements of Wilf-equivalence concerning permutation statistics.
In this paper, we consider the following permutation statistics for every permutation $\pi\in\S_n$:

\begin{itemize}[leftmargin=*,itemsep=0pt,topsep=3pt]
\item $\DT(\pi)=\{\pi_i\mid 1\le i<n,\ \pi_i>\pi_{i+1}\}$, the \emph{descent top set} of $\pi$,
\item $\AT(\pi)=\{\pi_{i+1}\mid 1\le i<n,\ \pi_i<\pi_{i+1}\}$, the \emph{ascent top set} of $\pi$,
\item $\Peak(\pi)=\{\pi_i\mid 1<i<n,\ \pi_{i-1}<\pi_i>\pi_{i+1}\}=\DT(\pi)\cap\AT(\pi)$, the \emph{peak set} of $\pi$,
\item $\PW(\pi)$, the word obtained by reading the peaks of $\pi$ from left to right.
\end{itemize}

This paper is motivated by a conjecture posed by Burstein concerning the classification of Destop-Wilf-equivalence classes for patterns of length $4$.

\begin{conjecture}[~\cite{Burstein}, Conjecture 11]\label{conj:destop}
The non-singleton $\DT$-Wilf-equivalence classes in $\mathcal{S}_4$ are:
\begin{gather*}
\{1243,3412\};\\
\{2413,1423\};\\
\{2143,3421\};\\
\{2314,3124\};\\
\{2431,3142,3241,4132\}.
\end{gather*}
\end{conjecture}

Zhou, Zang, and Yan~\cite{ZZY} proved refinements for $\{2413,1423\}$ and $\{3142,3241,4132\}$.
In this paper, we settle the cases $\{1243,3412\}$ and $\{2143,3421\}$.
We state our main results as the following refinements.

\begin{theorem}\label{2143-3421}
There exists a bijection $\phi:\S_n(2143)\rightarrow\S_n(3421)$ such that, for every $\pi\in\S_n(2143)$,
\[
(\DT,\AT,\PW)(\pi)=(\DT,\AT,rev(\PW))(\phi(\pi)),
\]
where $rev(w)$ denotes the reversal of the word $w$.

In particular, the patterns $2143$ and $3421$ are $(\DT,\AT,\Peak)$-Wilf-equivalent.
\end{theorem}

\begin{theorem}\label{1243-3412}
There exists a bijection $\psi:\S_n(1243)\rightarrow\S_n(3412)$ such that, for every $\pi\in\S_n(1243)$,
\[
(\DT,\AT,\PW)(\pi)=(\DT,\AT,rev(\PW))(\psi(\pi)).
\]
In particular, the patterns $1243$ and $3412$ are $(\DT,\AT,\Peak)$-Wilf-equivalent.
\end{theorem}

A \emph{Dumont permutation of the first kind}~\cite{Dumont,Burstein,BJ} is a permutation $\pi\in\S_{2n}$ such that $\DT(\pi)=\{2,4,\ldots,2n\}$.
Let $D^1_{2n}(\tau)$ denote the set of Dumont permutations of the first kind avoiding $\tau$.
As stated in~\cite{Burstein}, the following conjecture, together with other Wilf-equivalences of patterns of length $4$ on Dumont permutations of the first kind conjectured by Archer and Lauderdale, motivated Conjecture~\ref{conj:destop}.

\begin{conjecture}[~\cite{BJ}, Conjecture 5.1]\label{conj:dumont}
For every $n\ge0$,
\[
|D^1_{2n}(2143)|=|D^1_{2n}(3421)|.
\]
\end{conjecture}

The conjecture was presented at the open problem sessions of Permutation Patterns in 2016 and 2018~\cite{BJ}.
Since our bijections preserve the descent top set, they restrict to bijections on Dumont permutations of the first kind.
These restrictions prove Conjecture~\ref{conj:dumont} and give the Wilf-equivalence for $1243$ and $3412$.

\begin{corollary}\label{du1}
For every $n\ge0$,
\[
\begin{aligned}
|D^1_{2n}(2143)|&=|D^1_{2n}(3421)|,\\
|D^1_{2n}(1243)|&=|D^1_{2n}(3412)|.
\end{aligned}
\]
\end{corollary}

\begin{remark}
Very recently, while we were writing the proofs, Dong, Liu, and Zhao~\cite{DLZ} posted a preprint presenting a different bijection between $\S_n(2143)$ and $\S_n(3421)$ that preserves $\DT$ and reverses $\PW$, and of course also proves Conjecture~\ref{conj:dumont}.
Our bijection additionally preserves the ascent top set, thus settling the case $\{1243,3412\}$ of Conjecture~\ref{conj:destop}.
Their work prompted us to examine how our bijections affect the left-to-right order of the peaks.
\end{remark}

\section{The bijections}
We first construct a bijection
\[
\theta:\S_n(3412)\longrightarrow\S_n(3421)
\]
such that
\[
(\DT,\AT,\PW)(\pi)=(\AT,\DT,\PW)(\theta(\pi))
\]
for every $\pi \in \S_n(3412)$.
By combining $\theta$ with reverse on permutations, we build a bijection between $\S_n(2143)$ and $\S_n(3421)$ and a bijection between $\S_n(1243)$ and $\S_n(3412)$.
Our construction uses a modification of the block decomposition of Zhou, Zang, and Yan~\cite{ZZY}.

For a permutation $\pi \in \S_n$ and an integer $i\in[n]$, $\pi$ can be uniquely written in the following form:
\[
\pi=PA_0B_0A_1B_1\cdots A_kB_k,
\]
where
\begin{itemize}[leftmargin=*,itemsep=0pt,topsep=3pt]
\item each $A_j$ is a maximal consecutive subsequence of $\pi$ with the elements not greater than $i$;
\item each $B_j$ is a maximal consecutive subsequence of $\pi$ with the elements greater than $i$;
\item $P$ is the consecutive subsequence before $A_0$;
\item $i$ belongs to $A_0$.
\end{itemize}

The words $P$ and $B_k$ may be empty.
We call this process the \emph{block decomposition} of $\pi$ with respect to $i$.
Define
\[
\A(\pi,i)=(A_1,A_2,\cdots,A_k).
\]
We call each $A_j(1\le j\le k)$ a \emph{block} of $\A(\pi,i)$.

For two words $U$ and $V$, $U$ is said to be less(resp. greater) than $V$, denoted by $U<V$(resp. $U>V$) if all the entries of $U$ are less(resp. greater) than the entries of $V$.
We say $U$ and $V$ are \emph{comparable} if either $U<V$ or $U>V$.

Let $M=(M_1,M_2,\ldots,M_k)$ be a $k$-tuple in which each $M_i$ is a word.
$M$ is said to be \emph{comparable} if any two of these words are comparable under the relation above.
Moreover, $M$ is \emph{increasing} if $M_1<M_2<\cdots<M_k$.
This concerns the relative sizes of entire blocks while an \emph{increasing word} has its individual letters in increasing order.
We use the analogous conventions for decreasing order.
The empty word $\eps$ and one-letter words, as well as $k$-tuples with $k\le1$, are all regarded as both increasing and decreasing.

\begin{lemma}\label{avoidance}
Let $\pi \in \S_n$.
Then
\begin{enumerate}
\item $\pi \in \S_n(3412)$ if and only if the word $A_1A_2\cdots A_k$ is decreasing for every $1\le i\le n$.
\item $\pi \in \S_n(3421)$ if and only if every such word is increasing.
\end{enumerate}
\end{lemma}

\begin{proof}
Let $\pi \in \S_n(3412)$ and $1\le i\le n$.
Suppose that the word $A_1A_2\cdots A_k$ is not decreasing.
Then there exist $a \in A_j$ and $b\in A_l$ with $1\le j\le l \le k$ such that $a<b$.
Choose any element $c \in B_0$.
Then $icab$ forms a pattern $3412$, a contradiction.
Therefore, $A_1A_2\cdots A_k$ is decreasing.

Conversely, assume that the word $A_1A_2\cdots A_k$ is decreasing for every $1\le i\le n$.
We need to prove that $\pi \in \S_n(3412)$.
If not, there exist integers $p<q<s<t$ such that $\pi_p,\pi_q,\pi_s,\pi_t$ is an occurrence of $3412$.
Then in the block decomposition of $\pi$ with respect to $\pi_p$, we have $\pi_s\in A_j$ and $\pi_t\in A_l$ with $j\le l$.
As $\pi_s<\pi_t$, we obtain $A_1A_2\cdots A_k$ is not decreasing, contradicting the hypothesis.
Thus (i) holds.
The proof of (ii) is analogous, we omit it here.
\end{proof}

Now we define transformations $T_i$ on permutations that form the bijection $\theta$. 
Let $\pi\in \S_n$ and assume that $\A(\pi,i)$ is comparable for some $i\in[n]$.
Write the block decomposition of $\pi$ with respect to $i$ as 
\[
\pi=PA_0B_0A_1B_1\cdots A_kB_k.
\]
Define
\[
S_i(\pi)=PA_0B_0A'_1B_1A'_2B_2\cdots A'_kB_k,
\]
where $A'_1<A'_2<\cdots<A'_k$ are the increasing rearrangement of $\{A_1,A_2,\ldots,A_k\}$.

We furthermore write $A_0=LiR$, and define
\[
E_i(\pi)=PA_0^*B_0A_1B_1A_2B_2\cdots A_kB_k,
\]
where
\[
A_0^*=\begin{cases}
iL,&L\ne\eps,\ R=\eps,\\
Ri,&L=\eps,\ R\ne\eps,\\
LiR,&\text{otherwise}.
\end{cases}
\]
We note that $\A(\pi,i)=\A(E_i(\pi),i)$ as tuples.
The transformation $T_i$ is defined by $T_i=S_i\circ E_i(=E_i\circ S_i)$.

\begin{lemma}\label{continue}
Let $\pi \in \S_n$.
Suppose that $\A(\pi,i)$ is comparable and that each of its blocks is a decreasing word for some $1\le i\le n$.
Then, for every $1\le j< i$,
\[
\A(T_i(\pi),j)\subseteq\A(\pi,j)
\]
as sets of words.
More precisely, $\A(T_i(\pi),j)$ is empty if $j$ belongs to a block of $\A(\pi,i)$; otherwise $\A(T_i(\pi),j)=\A(\pi,j)$ (as sets).
\end{lemma}

\begin{proof}
Write the block decomposition of $\pi$ with respect to $i$ as
\[
\pi=PA_0B_0A_1B_1\cdots A_kB_k.
\]
It is easy to see that $\A(E_i(\pi),j)=\A(\pi,j)$ as tuples for any $1\le j< i$.
It remains to consider $S_i$.

\smallskip\noindent
\textbf{Case 1.} $j$ occurs to the right of $B_0$ in $\pi$.

Let $A_q$ be the block containing $j$, where $1\le q\le k$.
In $T_i(\pi)$, every element to the right of $A_q$ is greater than $j$.
Since the word $A_q$ is decreasing, the elements following $j$ within $A_q$ are therefore all less than $j$.
Consequently, $\A(T_i(\pi),j)=\emptyset$.

\smallskip\noindent
\textbf{Case 2.} $j$ occurs to the left of $B_0$ in $\pi$.

By applying $S_i$ to $\pi$, all the blocks of $\A(\pi,j)$ to the left of $B_0$ remain unchanged.
Since $j<i$, each block in $\A(\pi,j)$ to the right of $B_0$ is contained in some block of $\A(\pi,i)$.
$S_i$ only reorders the blocks of $\A(\pi,j)$ to the right of $B_0$, it follows that $\A(T_i(\pi),j)=\A(\pi,j)$ (as sets).
\end{proof}

For $\pi\in\S_n(3412)$, we define our map $\theta$ as $\theta(\pi)=(T_1\circ T_2\circ\cdots\circ T_n)(\pi)$.
\begin{example}
For $\pi=3564271\in\S_7(3412)$, the only nonidentity steps are
\[
3564271\xrightarrow{\ T_5\ }5361742\xrightarrow{\ T_4\ }5361724.
\]
\end{example}

To show $\theta$ is a bijection between $\S_n(3412)$ and $\S_n(3421)$, denote by
\[
\pi^{(n)}=\pi,\qquad
\pi^{(i-1)}=T_i(\pi^{(i)})\quad(i=n,n-1,\ldots,1).
\]
Thus $\theta(\pi)=\pi^{(0)}$.
\begin{lemma}\label{feasible}
Let $\pi \in\S_n(3412)$ and $0\le i\le n$.
Then $\pi^{(i)}$ is well-defined and $\A(\pi^{(i)},j)$ is comparable for any $1\le j\le i$.
Moreover, each block of $\A(\pi^{(i)},j)$ is a decreasing word.
\end{lemma}
\begin{proof}
We use induction on $i$.
If $i=n$, we have $\pi^{(n)}=\pi$.
Then the conclusion follows from Lemma~\ref{avoidance}(i).
Suppose that $\pi^{(i+1)}$ is well-defined and that, for every $1\le j\le i+1$, $\A(\pi^{(i+1)},j)$ is comparable and consists of decreasing words.
In particular, $\A(\pi^{(i+1)},i+1)$ is comparable.
Thus $\pi^{(i)}=T_{i+1}(\pi^{(i+1)})$ is well-defined by the definition of $T_{i+1}$.
Moreover, since every block of $\A(\pi^{(i+1)},i+1)$ is decreasing by the induction hypothesis, Lemma~\ref{continue} applies and gives $\A(\pi^{(i)},j)\subseteq \A(\pi^{(i+1)},j)$.
Then the result follows directly from the induction hypothesis, completing the proof.
\end{proof}
\begin{lemma}\label{structure}
Let $\pi\in\S_n(3412)$ and $0\le i\le n-1$.
The following statements hold for every $i+1\le j\le n$.
\begin{enumerate}
\item $\A(\pi^{(i)},j)$ is increasing.
\item Every block in $\A(\pi^{(i)},j)$ has the form $DU$, where $D$ is decreasing with elements at most $i$, and $U$ is increasing with elements greater than $i$.
Either $D$ or $U$ may be empty.
\end{enumerate}
\end{lemma}

\begin{proof}
We proceed by induction on $i$.
If $i=n-1$, then $j=n$.
Both assertions hold since $\A(\pi^{(n-1)},n)$ is empty.
Assume that $\A(\pi^{(i+1)},j)$ is increasing and every block of $\A(\pi^{(i+1)},j)$ has the form $DU$ for any $i+2\le j\le n$, where $D$ is decreasing with elements at most $i+1$, and $U$ is increasing with elements greater than $i+1$.
We proceed to prove that the assertion holds for $i$.

It is clear that $\A(\pi^{(i)},i+1)$ is increasing.
Since $\A(\pi^{(i)},i+1)=\A(\pi^{(i+1)},i+1)$ (as sets), Lemma~\ref{feasible} tells us that every block of $\A(\pi^{(i)},i+1)$ is decreasing with elements at most $i$.
Now we assume that $j>i+1$.
Let $\pi^{(i+1)}=PA_0B_0A_1B_1\cdots A_kB_k$ be the block decomposition of $\pi^{(i+1)}$ with respect to $j$.
We consider the following two cases.

\smallskip\noindent
\textbf{Case 1.} $i+1$ occurs to the right of $B_0$ in $\pi^{(i+1)}$.

In this case, $i+1$ belongs to $A_q$ for some $1\le q\le k$.
By the induction hypothesis, $A_1<\cdots<A_k$.
Moreover, we can write $A_q=(i+1)DU$, where $D<i+1$ is decreasing and $U>i+1$ is increasing.
Therefore, $\A(\pi^{(i+1)},i+1)$ is empty and
\[
\pi^{(i)}=PA_0B_0\cdots A_{q-1}B_{q-1}D(i+1)UB_q\cdots A_kB_k.
\]
Clearly, $\A(\pi^{(i)},j)=(A_1,A_2,\cdots,A_{q-1},D(i+1)U,A_{q+1},\cdots,A_k)$ is increasing.
It is also easy to see that every block of $\A(\pi^{(i)},j)$ has the form $D'U'$, where $D'$ is decreasing with elements at most $i$, and $U'>i$ is increasing.

\smallskip\noindent
\textbf{Case 2.} $i+1$ occurs to the left of $B_0$ in $\pi^{(i+1)}$.

Note that $E_{i+1}$ takes place to the left of $B_0$.
If every element in every block of $\A(\pi^{(i+1)},j)$ is greater than $i+1$, then $\A(\pi^{(i)},j)=\A(\pi^{(i+1)},j)$ as $k$-tuples and the assertions follow immediately.
Otherwise, let $h$ be the biggest index for which $A_h$ contains an element less than $i+1$.
Then 
\[
A_1<A_2<\cdots<A_{h-1}<i+1<A_{h+1}<\cdots<A_k.
\]
Write $A_h=D'U'$, where $D'<i+1$ is nonempty and decreasing, and $U'>i+1$ is increasing ($U'$ may be empty).
The words $A_1,A_2,\cdots,A_{h-1},D'$ occupy the last $h$ positions in $\A(\pi^{(i+1)},i+1)$.
In $\A(\pi^{(i)},i+1)$, denote the words occupying these positions by $C_1<C_2<\cdots<C_{h-1}<E$.
They are decreasing words with elements less than $i+1$.
Therefore, $\A(\pi^{(i)},j)=(C_1,C_2,\ldots,C_{h-1},EU',A_{h+1},\ldots,A_k)$.
The words $C_1,C_2,\cdots,C_{h-1},E$ are decreasing words with elements at most $i$ by Lemma~\ref{feasible} and the words $U',A_{h+1},\ldots,A_k$ are unchanged, so every block in $\A(\pi^{(i)},j)$ has the required form.

To show $\A(\pi^{(i)},j)$ is increasing, it suffices to show that $C_{h-1}<EU'<A_{h+1}$.
This is evident, completing the induction.
\end{proof}

\begin{lemma}\label{well}
The map $\theta$ is well-defined from $\S_n(3412)$ to $\S_n(3421)$.
\end{lemma}

\begin{proof}
Lemma~\ref{feasible} tells us that the map $\theta$ is feasible.
Set $i=0$ in Lemma~\ref{structure}, we have that $\A(\pi^{(0)},j)$ is increasing and every block of it is increasing.
Lemma~\ref{avoidance} (ii) gives $\theta(\pi)=\pi^{(0)}\in\S_n(3421)$.
\end{proof}

Next we proceed to show that $\theta$ is a bijection.
\begin{lemma}\label{recover}
Let $\pi \in\S_n(3412)$ and $1\le i\le n$.
Write the block decomposition of $\pi^{(i)}$ with respect to $i$ as $\pi^{(i)}=PA_0B_0A_1B_1\cdots A_kB_k$.
For $1\le t\le k$, let $G_t$ be the subword of $PB_0B_1\cdots B_{t-1}$ consisting of elements greater than $i$.
Then

\[A_t=
\begin{cases}
\max\{A_t,\ldots,A_k\},&\text{if $G_t$ is decreasing},\\
\min\{A_t,\ldots,A_k\},&\text{otherwise}.
\end{cases}
\]

\end{lemma}
\begin{proof}
The extrema exist by Lemma~\ref{feasible}.
By Lemma~\ref{feasible} and Lemma~\ref{continue}, we have that $\A(\pi^{(i)},i)\subseteq\A(\pi^{(i+1)},i)\subseteq\cdots\subseteq\A(\pi^{(n)},i)$ (as sets).
It follows that every $A_t(1\le t\le k)$ belongs to $\A(\pi^{(h)},i)$ for any $i\le h\le n$.

Fix $1\le t\le k$.
Suppose first that $G_t$ is decreasing.
We claim that $A_t$ is not contained in any block of $\A(\pi^{(j)},j)$ with $j>i$.
Otherwise, choose the smallest such $j$.
Then $\pi^{(j-1)}$ can be written in the form $\pi^{(j-1)}=\cdots j\cdots c\cdots A_t\cdots$, where $c>j$.
By the choice of $j$, every $S_l$ acts to the right of $A_t$, while $E_l$ moves only $l<j$, where $i<l<j$.
Hence $j$ and $c$ still occur to the left of $A_t$ in this order, contradicting the hypothesis that $G_t$ is decreasing.
Thus every $S_j$, $j>i$, acts to the right of $A_t$.
Since each $E_j$ moves only $j>i$, we have that $A_t$ precedes $A_q(t<q\le k)$ in $\pi$.
Lemma~\ref{avoidance}(i) therefore gives $A_t=\max\{A_t,A_{t+1},\cdots,A_k\}$.

Now suppose that $G_t$ contains an increasing pair $a<b$.
Then each of $A_t,A_{t+1},\cdots,A_k$ is contained in some block of $\A(\pi^{(i)},a)$.
It follows from Lemma~\ref{structure}(ii) that they are contained in different blocks.
By Lemma~\ref{structure}(i), $\A(\pi^{(i)},a)$ is increasing, implying that $A_t<A_{t+1}<\cdots<A_k$.
\end{proof}

\begin{lemma}\label{bijective}
The map $\theta$ is a bijection between $\S_n(3412)$ and $\S_n(3421)$.
\end{lemma}

\begin{proof}
By Lemma~\ref{well} and the well known fact that $|\S_n(3412)|=|\S_n(3421)|$~\cite{West}, it suffices to prove that $\theta$ is injective.
To this end, we show that $\pi^{(i-1)}$ uniquely determines $\pi^{(i)}$.

Write the block decomposition of $\pi^{(i-1)}$ with respect to $i$ as
\[
\pi^{(i-1)}=PA_0^*B_0A'_1B_1\cdots A'_kB_k.
\]
By the definition of $T_i$, $\pi^{(i)}$ is of the form 
\[
\pi^{(i)}=PA_0B_0A_1B_1\cdots A_kB_k,
\]
where $A_0$ is obtained directly from $E_i(\pi^{(i-1)})$ since $E_i$ is clearly an involution and $A_1, A_2,\cdots,A_k$ is an arrangement of the set $\mathcal{R}=\{A'_1, A'_2,\cdots, A'_k\}$. 

We obtain the word $G_1$ by choosing all elements greater than $i$ in $PB_0$ and reading them from left to right.
By Lemma~\ref{recover}, in $\pi^{(i)}$, $A_1=\max\mathcal{R}$ if $G_1$ is decreasing, and $A_1=\min\mathcal{R}$ otherwise.
Place that word in the position and remove it from $\mathcal{R}$.
For $t=2,3,\ldots,k$, repeat the same rule with the remaining set $\mathcal{R}$.
This determines $A_1,A_2,\ldots,A_k$, completing the proof.
\end{proof}

\begin{lemma}\label{sta}
Let $\pi \in \S_n(3412)$.
Then
\begin{enumerate}
  \item $\DT(\pi)=\AT(\theta(\pi))$;
  \item $\AT(\pi)=\DT(\theta(\pi))$;
  \item $\PW(\pi)=\PW(\theta(\pi))$.
\end{enumerate}
\end{lemma}

\begin{proof}
Let $i,j\in [n]$ with $i\neq j$.
Write the block decomposition of $\pi^{(i)}$ with respect to $i$ as $\pi^{(i)}=PA_0B_0A_1B_1\cdots A_kB_k$ and write $A_0=LiR$.

Assume that $j\in \DT(\pi^{(i)})$.
We proceed to show that $j\in \DT(\pi^{(i-1)})$ by considering three cases.
If $j$ occurs in $P$ or some $B_h(0\le h\le k)$, it is easy to see that $j\in \DT(\pi^{(i-1)})$.
If $j$ occurs in some $A_h$ with $1\le h\le k$, $j$ can not be the last element of $A_h$.
By the definition of $T_i$, we also have $j\in \DT(\pi^{(i-1)})$.
If $j$ occurs in $A_0$, then $j$ lies in $L$ or $R$ and is not the last element of that word.
It follows that $j\in\DT(\pi^{(i-1)})$.
Thus we have $j\in \DT(\pi^{(i)})\Rightarrow j\in\DT(\pi^{(i-1)})$.
Since the above process is reversible, we have that $j\in\DT(\pi^{(i-1)})\Rightarrow j\in\DT(\pi^{(i)})$.
Hence $j\in\DT(\pi^{(i)})\Leftrightarrow j\in\DT(\pi^{(i-1)})$.
Similarly, $j\in\AT(\pi^{(i)})\Leftrightarrow j\in\AT(\pi^{(i-1)})$.

The membership of $i$ in the two top sets is determined by whether $L$ and $R$ are empty:
\[
\begin{array}{ccc}
\toprule
\text{Condition}&\text{$i$ in $\pi^{(i)}$}&\text{$i$ in $\pi^{(i-1)}$}\\ \midrule
L\ne\eps,\ R=\eps&\text{ascent top only}&\text{descent top only}\\
L=\eps,\ R\ne\eps&\text{descent top only}&\text{ascent top only}\\
L\ne\eps,\ R\ne\eps&\text{both}&\text{both}\\
L=R=\eps&\text{neither}&\text{neither}
\\ \bottomrule
\end{array}
\]
Thus $i\in \DT(\pi^{(i)})\Leftrightarrow i\in \AT(\pi^{(i-1)})$ and $i\in \AT(\pi^{(i)})\Leftrightarrow i\in \DT(\pi^{(i-1)})$.
Combining the above arguments and the definition of $\theta$, we get (i) and (ii).

We proceed to prove (iii).
By Lemma~\ref{feasible}, $A_l(1\le l\le k)$ is decreasing and thus contains no peak.
A nonidentity $E_i$ moves only $i$ from one end of $A_0$ to the other, which in this case is clearly not a peak.
Thus $\PW((\pi)^{(i)})=\PW(\pi^{(i-1)})$, implying (iii).
\end{proof}

\begin{proof}[\normalfont\bfseries Proof of Theorem~\ref{2143-3421}]
By Lemma~\ref{bijective} and Lemma~\ref{sta}, the map $\phi=\theta\circ\rev$ is the bijection between $\S_n(2143)$ and $\S_n(3421)$ we want, completing the proof.
\end{proof}

\begin{proof}[\normalfont\bfseries Proof of Theorem~\ref{1243-3412}]
It is easy to see that $\psi=\rev\circ\phi^{-1}\circ\rev$ is the bijection between $\S_n(1243)$ and $\S_n(3412)$ we want.
\end{proof}


\begin{thebibliography}{99}
\bibitem{Bloom} J. Bloom, A refinement of Wilf-equivalence for patterns of length 4, J. Combin. Theory Ser. A \textbf{124} (2014) 166--177, \url{https://doi.org/10.1016/j.jcta.2014.01.001}.

\bibitem{BS1} J. Bloom, D. Saracino, On bijections for pattern-avoiding permutations, J. Combin. Theory Ser. A \textbf{116} (2009) 1271--1284, \url{https://doi.org/10.1016/j.jcta.2009.03.006}.

\bibitem{BS2} J. Bloom, D. Saracino, Another look at bijections for pattern-avoiding permutations, Adv. Appl. Math. \textbf{45} (2010) 395--409, \url{https://doi.org/10.1016/j.aam.2010.03.001}.

\bibitem{Bona} M. B\'{o}na, Combinatorics of Permutations, Chapman \& Hall/CRC, Boca Raton, FL, 2004, \url{https://doi.org/10.1201/9780203494370}.

\bibitem{Burstein} A. Burstein, Distribution of sets of descent tops and descent bottoms on restricted permutations, Discrete Math. Theor. Comput. Sci. \textbf{26}:1 (2025), Article 8, \url{https://doi.org/10.46298/dmtcs.12636}.

\bibitem{BJ} A. Burstein, O. Jones, Enumeration of Dumont permutations avoiding certain four-letter patterns, Discrete Math. Theor. Comput. Sci. \textbf{22}:2 (2021), Article 7, \url{https://doi.org/10.46298/dmtcs.6174}.

\bibitem{Chan} J.H.C. Chan, An infinite family of inv-Wilf-equivalent permutation pairs, Eur. J. Combin. \textbf{44} (2015) 57--76, \url{https://doi.org/10.1016/j.ejc.2014.08.031}.

\bibitem{CL} J.N. Chen, Z. Lin, A bijection for length-5 patterns in permutations, J. Combin. Theory Ser. A \textbf{202} (2024) 105815, \url{https://doi.org/10.1016/j.jcta.2023.105815}.

\bibitem{CK} A. Claesson, S. Kitaev, Classification of bijections between 321- and 132-avoiding permutations, S\'{e}m. Lothar. Combin. \textbf{60} (2008) B60d.

\bibitem{DRS} E. Deutsch, A. Robertson, D. Saracino, Refined restricted involutions, Eur. J. Combin. \textbf{28} (2007) 481--498, \url{https://doi.org/10.1016/j.ejc.2004.09.006}.

\bibitem{pps} T. Dokos, T. Dwyer, B.P. Johnson, B.E. Sagan, K. Selsor, Permutation patterns and statistics, Discrete Math. \textbf{312} (2012) 2760--2775, \url{https://doi.org/10.1016/j.disc.2012.05.014}.

\bibitem{DLZ} Y. Dong, L.L. Liu, T. Zhao, Descent tops and pinnacles on 2143- and 3421-avoiding permutations, preprint, arXiv:2610.01077 (2026), \url{https://doi.org/10.48550/arXiv.2610.01077}.

\bibitem{Dumont} D. Dumont, Interpr\'{e}tations combinatoires des nombres de Genocchi, Duke Math. J. \textbf{41} (1974) 305--318, \url{https://doi.org/10.1215/S0012-7094-74-04134-9}.

\bibitem{eli} S. Elizalde, Fixed points and excedances in restricted permutations, Electron. J. Combin. \textbf{18}(2) (2012) P29, \url{https://doi.org/10.37236/2025}.

\bibitem{Kitaev} S. Kitaev, Patterns in Permutations and Words, Monographs in Theoretical Computer Science. An EATCS Series, Springer, Berlin, Heidelberg, 2011, \url{https://doi.org/10.1007/978-3-642-17333-2}.
    
\bibitem{MacMahon} P.A. MacMahon, Combinatory Analysis, vol. 1, Cambridge University Press, London, 1915.

\bibitem{RSZ} A. Robertson, D. Saracino, D. Zeilberger, Refined restricted permutations, Ann. Comb. \textbf{6} (2002) 427--444, \url{https://doi.org/10.1007/s000260200015}.

\bibitem{West} J. West, Generating trees and the Catalan and Schr\"oder numbers, Discrete Math. \textbf{146} (1995) 247--262, \url{https://doi.org/10.1016/0012-365X(94)00067-1}.

\bibitem{YGZ} S.H.F. Yan, H. Ge, Y. Zhang, On a refinement of Wilf-equivalence for permutations, Electron. J. Combin. \textbf{22}(1) (2015) P1.20, \url{https://doi.org/10.37236/4465}.

\bibitem{YWZ} S.H.F. Yan, L.T. Wang, R.D.P. Zhou, On refinements of Wilf-equivalence for involutions, J. Algebraic Combin. \textbf{58} (2023) 69--94, \url{https://doi.org/10.1007/s10801-023-01239-1}.

\bibitem{ZY1} R.D.P. Zhou, S.H.F. Yan, Combinatorics of exterior peaks on pattern-avoiding symmetric transversals, Ann. Comb. \textbf{28} (2024) 977--1002, \url{https://doi.org/10.1007/s00026-023-00664-0}.

\bibitem{ZY2} R.D.P. Zhou, S.H.F. Yan, Equidistribution of set-valued statistics on standard Young tableaux and transversals, Adv. Appl. Math. \textbf{155} (2024) 102669, \url{https://doi.org/10.1016/j.aam.2023.102669}.

\bibitem{ZZY} R.D.P. Zhou, Y. Zang, S.H.F. Yan, Further refinements of Wilf-equivalence for patterns of length 4, J. Combin. Theory Ser. A \textbf{204} (2024) 105863, \url{https://doi.org/10.1016/j.jcta.2024.105863}.
\end{thebibliography}
\end{document}